\documentclass[a4paper,12pt]{article}
\usepackage{amssymb,amsthm,amsmath,amsfonts}
\usepackage{graphicx}
\usepackage[left=2.5cm,right=2.5cm,top=2cm,bottom=2cm]{geometry}
\usepackage{cite}
\usepackage{color}

 \newtheorem{thm}{Theorem}[section]

\theoremstyle{definition}
\newtheorem{defn}[thm]{Definition}
\theoremstyle{remark}
\newtheorem{rem}[thm]{Remark}
\newtheorem*{ex}{Example}
\numberwithin{equation}{section}

\begin{document}
	
\title{A Non-integer Dimensional Space Approach to the Moisil-Teodorescu Operator}
\small{
\author{
Juan Bory-Reyes$^{1}$, Marco Antonio P\'erez-de la Rosa$^{2}$,\\ Jos\'e Oscar González-Cervantes$^{3}$, Juan Eduardo Napoles-Valdes$^{4{\footnote{corresponding author}}}$}}
\vskip 1truecm
\date{\small $^{1}$ SEPI-ESIME-IPN-Zacatenco, Instituto Polit\'ecnico Nacional, 07738, Ciudad M\'exico, M\'exico \\ Email: juanboryreyes@yahoo.com\\ $^{2}$ Department of Actuarial Sciences, Physics and Mathematics. Universidad de las Am\'{e}ricas Puebla 72810, Puebla. Mexico\\ Email: marco.perez@udlap.mx\\ $^{3}$ Departamento de Matem\'aticas, ESFM-Instituto Polit\'ecnico Nacional. 07338, Ciudad M\'exico, M\'exico\\ Email: jogc200678@gmail.com\\ $^{4}$ Facultad de Ciencias Exactas y Naturales y Agrimensura. Universidad Nacional del Nordeste. 3400, Corrientes, Argentina\\ Email: profjnapoles@gmail.com}

\maketitle

\begin{abstract}
The vector calculus in non-integer dimensional space (NIDS), including the NIDS version of the standard vector differential operators (gradient, divergence, and curl) is well-known. A deformation of the quaternionic Moisil-Teodorescu operator, written in terms of NIDS derivatives is the main purpose of this article. Along similar lines, we consider the NIDS reformulation of the quaternionic Bitsadze operator and the Lam\'{e}-Navier operator of the three-dimensional elasticity theory. Also, a quaternionic reformulation of a NIDS time-harmonic Maxwell system is introduced, whose solutions are directly related with those of the perturbed NIDS Moisil-Teodorescu operator. Finally, a generalized approach to the study is addressed.
\end{abstract}

\textbf{Math Subject Classification (2020):} 30G35, 31B10, 74B05.

\section{Introduction and Basic definitions}
In 2014, Khalil, Al Horani, Yousef and Sababheh \cite{KHYS} define a local kind derivative called by the authors “conformable fractional” derivative, depending just on the basic limit definition of the derivative. Indeed, for $f:[0,\infty)\to\mathbb{R}$ the conformable fractional derivative of order $\alpha$, with $0<\alpha<1$ was defined as
\begin{equation}\label{cdo}
(D_{\alpha}f)(x)=\lim_{h\to 0}\frac{f\left(x+x^{1-\alpha}h\right)-f(x)}{h},
\end{equation}
for all $x>0$ and the derivative at $0$ is defined as 
\begin{equation}
(D_{\alpha}f)(0) = \lim_{x\to 0+}(D_{\alpha}f)(x).
\end{equation} 
It may be worth reminding the reader that the name "fractional-like" instead "conformable fractional" derivative has appeared in the literature, see for instance \cite{MaSt}. 

We follow \cite{ABA} just to say that conformable fractional derivative is merely a generalized $q$-derivative (Jackson's-derivative \cite{J}) for $q = 1+hx^{-\alpha}$, when $h\rightarrow 0$ and $q\rightarrow 1$. 

The conformable fractional derivative operator has recently garnered significant attention from scientists and has discussed in several works. The advantages geometric and physical implications of the conformable derivative is well-known. In \cite{AARB, ABA, H, ABD} and the bibliography therein, the reader will be quoted of some illustrative examples. 

Several works shown directly and indirectly that the “conformable fractional” derivatives suggested in 2014 cannot be considered as derivatives of non-integer orders, nor as fractional derivatives, see \cite{Ab, A, AM, T1}. The “conformable fractional” derivatives are operators of the integer orders, which where proposed and applied since at least in 2005-2014 to describe fractal media. This is the reason, why the term “conformable” is advised instead of the introduced expression “conformal fractional” to refer to the operator associated with (\ref{cdo}). As was suggested by Tarasov in \cite{T1}, conformable operators should be identify with operators in non-integer dimensional spaces (NIDS) in continuous models with power-law density of states (DOS) and therefore it can be called NIDS operators.

Throughout this paper, we shall use the expression ‘‘NIDS’’ in operator associated to the conformable derivative.  

Quaternionic analysis (functions in the kernel of the first-order elliptic (but not strongly elliptic) Moisil-Theodoresco operator \cite{MT}) has became a well established branch in mathematics and greatly successful in many different directions. Much work has been done to develop this extension of complex analysis focuses on the connection between analysis, geometry with the algebraic structure of quaternions, especially by the Swiss mathematician Fueter \cite{Fu} (summarized in English by Deavours \cite{De} and expanded by Sudbery \cite{Sub}) with fair success. A thorough account of the theory can be found in \cite{GuSp}.

The Moisil-Theodoresco operator is nowadays considered to be a good analogue of the usual Cauchy-Riemann operator of complex analysis to the quaternionic setting and it is a square root of the scalar Laplace operator in $\mathbb{R}^3$, which can be seen as a generalization of the Bitsadze operator to the space $\mathbb{R}^3$ and plays an interesting role in the treatment of several problems from mathematical physics with quaternionic analysis techniques, in particular the Lamé-Navier system from the theory of elasticity, see \cite{B, Di, Dzhuraev}. 

The success and application of conformable calculus is a strong motivation to show its connection with quaternionic analysis. In particular, the main goal of the present work is to develop a quaternionic structure for a novel conformable Moisil-Teodorescu operator, to be named NIDS Moisil-Teodorescu operator, which is obtained after doing so for the ordinary divergence, gradient and curl operators. To the best of our knowledge this work is the first investigation of results in quaternionic analysis in the setting of the conformable calculus theory and we consider that it constitutes the groundwork for building a complete and comprehensive new function theory.

The work is organized as follows: After this brief introduction, Section 2 provides the foundational concepts and tools needed. In particular, contains the definition of the NIDS standard vector differential operators (gradient, divergence, and curl), together with some of their relations. In Section 3, we introduce the NIDS Moisil-Teodorescu operator besides the corresponding Laplacian and the NIDS version of the Bitsadze operator as well as a NIDS Lam\'{e}-Navier operator equation for fractal media. In Section 5, a quaternionic reformulation of a NIDS time-harmonic Maxwell system is introduced, whose solutions are directly related with those of the perturbed NIDS Moisil-Teodorescu operator. Finally, further enhancements through a generalized approach of the study is proposed. 
\section{Preliminaries}
In this section we provide a brief exposition of foundational concepts and tools developed in \cite{T1, T2} according to our purpose, which serve as the basis for our subsequent advancements. For a treatment of a more general case we refer the reader to \cite{T3, T4}.
\subsection{Generalization of gradient, divergence and curl operators to non-integer dimensional space}

In order to describe fractal media and distributions by the continuum models with non-integer dimensional space (NIDS), we should use the concepts of fractal density of states $c_n\left(D,\vec{x}\right)$ and distribution function $f\left(\vec{x}\right)$, where $\vec{x}\in\mathbb{R}^n$ ($n=3$ in our case). For systems with $n$ degrees of freedom, we can consider $\mathbb{R}^n$ as a configuration space, in which fractal distributions of particle states are described. The fractal density of states (FDOS) describes how closely packed permitted states (or places) in the space $\mathbb{R}^n$, where the particles are distributed. 

The expression $c_n\left(D,\vec{x}\right)dV_v$ is equal to the number of permitted states between $V_n$ and $V_n+dV_n$ in fractal media that is distributed in the space $\mathbb{R}^n$, $n\in\mathbb{N}$ with physical dimension $D\in (0,n]$. The distribution function $f\left(\vec{x}\right)$ describes distribution of physical values such as mass, electric charge, number of particles on a set of permitted states. In general, we cannot reduce all properties of the fractal medium to just the distribution function, and we should use the density of states (DOS) concept. The concepts of DOS and NIDS are well-known and actively used in various areas of theoretical and applied physics. 

The single-variable measures are 
\[d\mu_1(\alpha_k,x_k)=c_1(\alpha_k,x_k)dx_k,\quad k=1,\dots,n,\]
where
\[c_1(\alpha_k,x_k)=\frac{\pi^{\alpha_k/2}}{\Gamma(\alpha_k/2)}|x_k|^{\alpha_k-1},\]
with $\Gamma$ the Gamma function and $\alpha_k\in(0,1]$ such that $D=\sum_{k=1}^{n}\alpha_k$ is interpreted as a dimension with $D\in(0, n]$. If all $\alpha_k=1$, we get $D=n$, i.e., the dimension of space is the standard integer dimension. If all $\alpha_k=\alpha$, where $0<\alpha<1$, we have a NIDS for isotropic fractal media and $D=n\alpha$. In general, we have a NIDS, if at least one of the parameters $\alpha_k$ is not equal to $1$.

The single-variable measures can be mathematically interpreted as the volume of an infinitesimally small $\alpha_k$-dimensional ball, where $|x_k|=R_k$ is considered as a radius of this $\alpha_k$-dimensional ball. The functions $c_1(\alpha_k,x_k)$ should be considered as a density of states of fractal media of distributions along the $X_k$-axis. The DOS $c_1(\alpha_k,x_k)$ describes how closely packed permitted states of matter particles in the space $\mathbb{R}$.

For $n=1$, the first-order derivative in NIDS with the dimension $\alpha\in(0,1]$ is defined as:
\[\left(D^{T,\alpha}f\right)(x)=\frac{1}{c_1(\alpha,x)}\frac{\partial}{\partial x}f(x)=\frac{\Gamma(\alpha/2)}{\pi^{\alpha/2}}|x|^{1-\alpha}\frac{\partial}{\partial x}f(x).\]

We now indicate an expression for the Nabla-operator in NIDS.
\[\nabla_{(\alpha)} = \sum_{k=1}^{n} \mathbf{e}_{k} \frac{1}{c_1(\alpha_k,x_k)}\displaystyle \frac{\partial}{\partial x_{k}},\]
where $\mathbf{e}_{k}$, $k= 1,\dots,n$ denotes the standard orthonormal basis of $\mathbb{R}^n$ and $(\alpha):=(\alpha_1,\dots,\alpha_n).$

The NIDS standard vector differential operators with $n=3$ and $D\in (0, 3]$ are expressed through the DOS as follows:
\begin{defn}
Let $\Omega$ be a domain in $\mathbb{R}^3$ and $f\in C^{1}\left(\Omega, \mathbb{R}\right)$. The NIDS gradient is defined as
\begin{equation}\label{nids_grad}
	\mathrm{grad}_{(\alpha)}[f]\left(\vec{x}\right):= \nabla_{(\alpha)}[f]  = \sum_{k=1}^{3}\mathbf{e}_{k}\frac{1}{c_1(\alpha_k,x_k)}\frac{\partial f}{\partial x_k}\left(\vec{x}\right).
\end{equation}
\end{defn}

\begin{defn}
Let $\Omega$ be a domain in $\mathbb{R}^3$ and $\vec g \in C^{1}\left(\Omega, \mathbb{R}^3\right)$. The NIDS  divergence of order $\alpha$ of $\mathbf{F}$ is defined as
\begin{equation}\label{nids_div}
\mathrm{div}_{(\alpha)}\left[\vec g \right]\left(\vec{x}\right):=  \nabla_{(\alpha)} \cdot [\vec g ] = \sum_{k=1}^{3}\frac{1}{c_1(\alpha_k,x_k)}\frac{\partial g_k}{\partial x_k}\left(\vec{x}\right),
\end{equation}
where $g_1, \, g_2, \, g_3$  are the components of $\vec g$.
\end{defn}

\begin{defn}
Let $\Omega$ be a domain in $\mathbb{R}^3$ and $\vec g\in C^{1}\left(\Omega, \mathbb{R}^3\right)$. The NIDS  curl of $\vec g $ is defined as
\begin{equation}\label{nids_curl}
\mathrm{curl}_{(\alpha)}\left[\vec g  \right]\left(\vec{x}\right):= \nabla_{(\alpha)}\times [\vec g]  =  \sum_{i,j,k=1}^{3}\mathbf{e}_{i}\epsilon_{ijk}\frac{1}{c_1(\alpha_j,x_j)}\frac{\partial g_k}{\partial x_j}\left(\vec{x}\right),
\end{equation}
where $\epsilon_{ijk}$ is the Levi-Civita symbol.
\end{defn}

The NIDS vector differential operators that we recalled can be applied in ways that lead to a large number of identities, in particular let us show the following result.
\begin{thm}\label{Combinations of conformable operators}
Let $\Omega$ be a domain in $\mathbb{R}^3$, $f\in C^{2}\left(\Omega, \mathbb{R}\right)$ and $\vec g \in C^{2}\left(\Omega, \mathbb{R}^3\right)$. The NIDS operators satisfy the following relations
	\begin{align*}
		\mathrm{div}_{(\alpha)}\left[\mathrm{grad}_{(\alpha)}\left[f\right]\right]&=\sum_{k=1}^{3}\frac{1}{c_{1}^{2}(\alpha_k,x_k)}\left(\frac{\partial^2 f}{\partial x_{k}^2}-\frac{\alpha_k-1}{x_k}\frac{\partial f}{\partial x_k}\right),\\
		\mathrm{curl}_{(\alpha)}\left[\mathrm{grad}_{(\alpha)}\left[f\right]\right]&=0,\\
		\mathrm{grad}_{(\alpha)}\left[\mathrm{div}_{(\alpha)}\left[\vec g\right]\right]&=\sum_{k=1}^{3}\mathbf{e}_k\frac{1}{c_{1}(\alpha_k,x_k)}\left(\sum_{\ell=1}^{3}\frac{1}{c_{1}(\alpha_\ell,x_\ell)}\frac{\partial^2 g_\ell}{\partial x_{k}\partial x_\ell}\right.\\
		&\qquad\left.-\frac{1}{c_{1}(\alpha_k,x_k)}\frac{\alpha_k-1}{x_k}\frac{\partial g_k}{\partial x_k}\right),\\
		\mathrm{div}_{(\alpha)}\left[\mathrm{curl}_{(\alpha)}\left[\vec g  \right]\right]&=0,\\
		\mathrm{curl}_{(\alpha)}\left[\mathrm{curl}_{(\alpha)}\left[\vec g \right]\right]&=\sum_{k=1}^{3}\mathbf{e}_k\left(\sum_{\ell=1,\,\ell\neq k}^{3}\left(\frac{1}{c_{1}(\alpha_\ell,x_\ell)c_{1}(\alpha_k,x_k)}\frac{\partial^2 g_\ell}{\partial x_{\ell}\partial x_k}\right.\right.\\
		&\qquad\left.\left.-\frac{1}{c_{1}^{2}(\alpha_\ell,x_\ell)}\left(\frac{\partial^2 g_k}{\partial x_{\ell}^2}-\frac{\alpha_\ell-1}{x_\ell}\frac{\partial g_k}{\partial x_\ell}\right)\right)\right).
	\end{align*}
\end{thm}
\begin{proof}
Is a direct application of the previous definitions.
\end{proof}
\subsection{Basic tools of quaternionic analysis}
In this section, we go over some basic facts of quaternionic analysis, limited to those used in this work. For a more complete description of the theory we refer the reader to \cite{GuSp}.

The skew-field of real quaternions $\mathbb H$ (a non-commutative division ring) is used in this work. Each $a\in \mathbb H$ is of the form $a = \sum_{k=0}^3 a_k\mathbf{e}_k$, with $\{a_k\}\subset\mathbb R$; $\mathbf{e}_0 = 1$ and $\mathbf{e}_1, \mathbf{e}_2, \mathbf{e}_3$ represent the quaternionic imaginary units such  that  $e_1^2=e_2^2=e_3^2=e_1e_2e_3= -1$. 

A quaternion $a\in\mathbb H$ can be written as $a = a_0 + \vec{a}$, where $a_0$ and $\vec{a} := \sum_{k=1}^3 a_k\mathbf{e}_k$ will be called the scalar and vector part of $a$ respectively. If  $a_0=0$, then the quaternion $a$ can be identified with an ordinary vector of $\mathbb R^3$. 

For any $a,b\in\mathbb H$, their product can be written as:
\[a\,b = a_0\,b_0 - \left\langle\vec{a} , \vec{b}\right\rangle + a_0\,\vec{b} + b_0\,\vec{a} + \left[\vec{a} , \vec{b}\right],\]
where
\[\left\langle\vec{a} , \vec{b}\right\rangle := \sum_{k=1}^3 a_k\,b_k,\,\,\,\,\,\, \left[\vec{a} , \vec{b}\right]:=
\left |
\begin{array}{rrr}
	\mathbf{e}_1 & \mathbf{e}_2 & \mathbf{e}_3\\
	a_1& a_2 & a_3\\
	b_1& b_2 & b_3
\end{array}
\right |.
\]
Specifically, if the scalar part of $a$ and $b$ are both zero, then $a\,b :=  - \langle\vec a,\vec b\rangle + [\vec a,\vec b]$.

The Moisil-Teodorescu operator, $D_{MT}$, is defined as follows:
\begin{equation*}
	D_{MT}[f]:=\mathbf{e}_1\frac{\partial f}{\partial x_1}+\mathbf{e}_2\frac{\partial f}{\partial x_2}+\mathbf{e}_3\frac{\partial f}{\partial x_3}.
\end{equation*}

Note that $D_{MT}$ factorizes the scalar Laplacian and bears a resemblance to $\nabla$. It does, in fact, satisfy the following relationship: 
\begin{equation*}
	-D_{MT}^2=\Delta_{\mathbb{R}^3},
\end{equation*}
that indicates a number of advantages in terms of applications to physical problems. 

The behavior of the operator $D_{MT}$ on a $\mathbb{H}$-valued function $f:=f_0+ \vec{f}$, where $\vec{f}:=f_1\mathbf{e}_1+f_2\mathbf{e}_2+f_3\mathbf{e}_3$ can be described as follows:
\begin{equation*}
	D_{MT}[f]=-\mathrm{div}\left[\vec{f}\right]+\mathrm{grad}[f_0]+\mathrm{curl}\left[\vec{f}\right].
\end{equation*}
The preceding relationship can be deduced directly from the quaternionic product's properties. 

When the Moisil-Teodorescu operator is acting on the right, $D_{MT}^r[f]$ stands for:
\begin{equation*}
	D_{MT}^r[f]:=\frac{\partial f}{\partial x_1}\mathbf{e}_1+\frac{\partial f}{\partial x_2}\mathbf{e}_2+\frac{\partial f}{\partial x_3}{e}_3=-\mathrm{div}\left[\vec{f}\right]+\mathrm{grad}[f_0]-\mathrm{curl}\left[\vec{f}\right].
\end{equation*}
\section{A NIDS Moisil-Teodorescu operator}
We now introduce the NIDS Moisil-Teodorescu operator and its right version, acting on quaternionic valued functions defined over domains of $\mathbb{R}^3$.

\begin{equation}\label{conformalMT}
	D_{MT}^{(\alpha)}[f]=-\mathrm{div}_{(\alpha)}\left[\vec{f}\right]+\mathrm{grad}_{(\alpha)}\left[f_0\right]+\mathrm{curl}_{(\alpha)}\left[\vec{f}\right]
	\end{equation}
	and
	\begin{equation}\label{ConformalMT_right}
	D_{MT}^{(\alpha), r}[f]=-\mathrm{div}_{(\alpha)}\left[\vec{f}\right]+\mathrm{grad}_{(\alpha)}\left[f_0\right]-\mathrm{curl}_{(\alpha)}\left[\vec{f}\right],
\end{equation}
where $f=f_0+ \vec{f}$ is a quaternionic-valued function.

The quaternionic NIDS Laplacian $\Delta_{\mathbb H}^{(\alpha)}$ is factorized by the NIDS Moisil-Teodorescu operator $D_{MT}^{(\alpha)}$ as follows: 
\begin{equation*}
	\left(D_{MT}^{(\alpha)}\right)^2\left[f\right]=-\Delta_{\mathbb H}^{(\alpha)}\left[f\right],
\end{equation*}
where
\begin{equation*}
	\Delta_{\mathbb H}^{(\alpha)}[f]:=\Delta_0^{(\alpha)}\left[f_0\right]+\vec{\Delta}^{(\alpha)}\left[\vec{f}\right],
\end{equation*}
with
\begin{equation*}
\Delta_0^{(\alpha)}\left[f_0\right]:=\mathrm{div}_{(\alpha)}\left[\mathrm{grad}_{(\alpha)}\left[f_0\right]\right]+\mathrm{div}_{(\alpha)}\left[\mathrm{curl}_{(\alpha)}\left[\vec{f}\right]\right] 
\end{equation*}
and
\begin{equation*}
\vec{\Delta}^{(\alpha)}\left[\vec{f}\right]:=\mathrm{grad}_{(\alpha)}\left[\mathrm{div}_{(\alpha)}\left[\vec{f}\right]\right]-\mathrm{curl}_{(\alpha)}\left[\mathrm{grad}_{(\alpha)}\left[f_0\right]\right]-\mathrm{curl}_{(\alpha)}\left[\mathrm{curl}_{(\alpha)}\left[\vec{f}\right]\right].
\end{equation*}

From Theorem \ref{Combinations of conformable operators} one obtains the explicit form of the previous operators, as follows:
\begin{equation}\label{scalar_conformal_Laplacian}
	\Delta_{0}^{(\alpha)}\left[f_0\right]=\sum_{k=1}^{3}\frac{1}{c_{1}^{2}(\alpha_k,x_k)}\left(\frac{\partial^2 f_0}{\partial x_{k}^2}-\frac{\alpha_k-1}{x_k}\frac{\partial f_0}{\partial x_k}\right)
\end{equation}
and
\begin{align}
	\vec{\Delta}^{(\alpha)}\left[\vec{f}\right]&=\sum_{k=1}^{3}\mathbf{e}_k\frac{1}{c_{1}(\alpha_k,x_k)}\left(\sum_{\ell=1}^{3}\frac{1}{c_{1}(\alpha_\ell,x_\ell)}\frac{\partial^2 f_\ell}{\partial x_{k}\partial x_\ell}-\frac{1}{c_{1}(\alpha_k,x_k)}\frac{\alpha_k-1}{x_k}\frac{\partial f_k}{\partial x_k}\right)\notag\\
	&\quad-\sum_{k=1}^{3}\mathbf{e}_k\left(\sum_{\ell=1,\,\ell\neq k}^{3}\left(\frac{1}{c_{1}(\alpha_\ell,x_\ell)c_{1}(\alpha_k,x_k)}\frac{\partial^2 f_\ell}{\partial x_{\ell}\partial x_k}\right.\right.\notag\\
	&\qquad\qquad\qquad\left.\left.-\frac{1}{c_{1}^{2}(\alpha_\ell,x_\ell)}\left(\frac{\partial^2 f_k}{\partial x_{\ell}^2}-\frac{\alpha_\ell-1}{x_\ell}\frac{\partial f_k}{\partial x_\ell}\right)\right)\right).\label{vector_conformal_Laplacian}
\end{align}
The NIDS Laplace operators reduce to the standard Laplacian in $\mathbb{R}^3$, when all $\alpha_k=1$ in (\ref{scalar_conformal_Laplacian}) and (\ref{vector_conformal_Laplacian}).

\subsection{The NIDS Bitsadze operator}
Our purpose in this subsection is to derive the NIDS quaternionic Bitsadze operator, defined as: 
\begin{equation*}\label{cfqBo}
	\widetilde{\Delta_{\mathbb H}^{(\alpha)}}\left[f\right]:=\Delta_0^{(\alpha)}\left[f_0\right]+\widetilde{\vec{\Delta}}^{(\alpha)}\left[\vec{f}\right],
\end{equation*}
where
\begin{equation*}
\widetilde{\vec{\Delta}}^{(\alpha)}\left[\vec{f}\right]:=\mathrm{grad}_{(\alpha)}\left[\mathrm{div}_{(\alpha)}\left[\vec{f}\right]\right]+\mathrm{curl}_{(\alpha)}\left[\mathrm{grad}_{(\alpha)}\left[f_0\right]\right]+\mathrm{curl}_{(\alpha)}\left[\mathrm{curl}_{(\alpha)}\left[\vec{f}\right]\right].
\end{equation*}
Notice that $D_{MT}^{(\alpha)}$ and $D_{MT}^{(\alpha),r}$ factorize the  quaternionic Bitsadze operator $\widetilde{\Delta_{\mathbb H}^{(\alpha)}}$ as follows:
\begin{equation*}
	D_{MT}^{(\alpha),r} D_{MT}^{(\alpha)}\left[f\right]=-\widetilde{\Delta_{\mathbb H}^{(\alpha)}}\left[f\right].
\end{equation*}
From Theorem \ref{Combinations of conformable operators} one gets, the following explicit form of the quaternionic Bitsadze operator:

\begin{align}
\widetilde{\vec{\Delta}}^{(\alpha)}\left[\vec{f}\right]&=\sum_{k=1}^{3}\mathbf{e}_k\frac{1}{c_{1}(\alpha_k,x_k)}\left(\sum_{\ell=1}^{3}\frac{1}{c_{1}(\alpha_\ell,x_\ell)}\frac{\partial^2 f_\ell}{\partial x_{k}\partial x_\ell}-\frac{1}{c_{1}(\alpha_k,x_k)}\frac{\alpha_k-1}{x_k}\frac{\partial f_k}{\partial x_k}\right)\notag\\
	&\quad+\sum_{k=1}^{3}\mathbf{e}_k\left(\sum_{\ell=1,\,\ell\neq k}^{3}\left(\frac{1}{c_{1}(\alpha_\ell,x_\ell)c_{1}(\alpha_k,x_k)}\frac{\partial^2 f_\ell}{\partial x_{\ell}\partial x_k}\right.\right.\notag\\
	&\qquad\qquad\qquad\left.\left.-\frac{1}{c_{1}^{2}(\alpha_\ell,x_\ell)}\left(\frac{\partial^2 f_k}{\partial x_{\ell}^2}-\frac{\alpha_\ell-1}{x_\ell}\frac{\partial f_k}{\partial x_\ell}\right)\right)\right). \label{vector_Bitsadze_operator}
\end{align}

Note that if all $\alpha_k=1$ in (\ref{vector_Bitsadze_operator}), then the classical Bitsadze operator in $\mathbb{R}^3$ appears.

\subsection{NIDS quaternionic Helmholtz equation}

Consider the NIDS quaternionic Helmholtz equation:
\begin{equation}\label{quaternionic Helmholtz}
	\Delta_{\mathbb H}^{(\alpha)}[f]+\lambda^2 f=0,\quad \lambda\in\mathbb{C}.
\end{equation}

We have the following factorization of the previously defined operator:
\begin{equation*}
	-\left(D_{MT}^{(\alpha)}-\lambda{\mathcal I}\right)\left(D_{MT}^{(\alpha)}+\lambda{\mathcal I}\right)=\Delta_{\mathbb H}^{(\alpha)}+\lambda^2{\mathcal I},
\end{equation*}
where $\mathcal I$ denotes the identity operator.

From the previous factorization, the null solutions of the operator $D_{MT}^{(\alpha)}+\lambda{\mathcal I}$, a perturbation of the NIDS Moisil-Theodorescu operator, are particular solutions of Equation (\ref{quaternionic Helmholtz}). 

We get the component form of the NIDS quaternionic Helmholtz equation by putting Eqs. (\ref{scalar_conformal_Laplacian}) and (\ref{vector_conformal_Laplacian}) into Eq. (\ref{quaternionic Helmholtz}):

\begin{align}
	0&=\left[\frac{1}{c_{1}^{2}(\alpha_1,x_1)}\left(\frac{\partial^2 f_0}{\partial x_{1}^2}-\frac{\alpha_1-1}{x_1}\frac{\partial f_0}{\partial x_1}\right)+\frac{1}{c_{1}^{2}(\alpha_2,x_2)}\left(\frac{\partial^2 f_0}{\partial x_{2}^2}-\frac{\alpha_2-1}{x_2}\frac{\partial f_0}{\partial x_2}\right)\right.\notag\\
	&\left.\qquad+\frac{1}{c_{1}^{2}(\alpha_3,x_3)}\left(\frac{\partial^2 f_0}{\partial x_{3}^2}-\frac{\alpha_3-1}{x_3}\frac{\partial f_0}{\partial x_3}\right)+\lambda^2 f_0\right]\notag\\
	&+\left[\frac{1}{c_{1}(\alpha_1,x_1)}\left(\sum_{\ell=1}^{3}\frac{1}{c_{1}(\alpha_\ell,x_\ell)}\frac{\partial^2 f_\ell}{\partial x_{1}\partial x_\ell}-\frac{1}{c_{1}(\alpha_1,x_1)}\frac{\alpha_1-1}{x_1}\frac{\partial f_1}{\partial x_1}\right)\right.\notag\\
	&\left.\qquad-\sum_{\ell=1,\,\ell\neq 1}^{3}\left(\frac{1}{c_{1}(\alpha_\ell,x_\ell)c_{1}(\alpha_1,x_1)}\frac{\partial^2 f_\ell}{\partial x_{\ell}\partial x_1}-\frac{1}{c_{1}^{2}(\alpha_\ell,x_\ell)}\left(\frac{\partial^2 f_1}{\partial x_{\ell}^2}-\frac{\alpha_\ell-1}{x_\ell}\frac{\partial f_1}{\partial x_\ell}\right)\right)+\lambda^2 f_1\right]\,\mathbf{e}_1\notag\\
	&+\left[\frac{1}{c_{1}(\alpha_2,x_2)}\left(\sum_{\ell=1}^{3}\frac{1}{c_{1}(\alpha_\ell,x_\ell)}\frac{\partial^2 f_\ell}{\partial x_{2}\partial x_\ell}-\frac{1}{c_{1}(\alpha_2,x_2)}\frac{\alpha_2-1}{x_2}\frac{\partial f_2}{\partial x_2}\right)\right.\notag\\
	&\left.\qquad-\sum_{\ell=1,\,\ell\neq 2}^{3}\left(\frac{1}{c_{1}(\alpha_\ell,x_\ell)c_{1}(\alpha_2,x_2)}\frac{\partial^2 f_\ell}{\partial x_{\ell}\partial x_2}-\frac{1}{c_{1}^{2}(\alpha_\ell,x_\ell)}\left(\frac{\partial^2 f_2}{\partial x_{\ell}^2}-\frac{\alpha_\ell-1}{x_\ell}\frac{\partial f_2}{\partial x_\ell}\right)\right)+\lambda^2 f_2\right]\,\mathbf{e}_2\notag\\
	&+\left[\frac{1}{c_{1}(\alpha_3,x_3)}\left(\sum_{\ell=1}^{3}\frac{1}{c_{1}(\alpha_\ell,x_\ell)}\frac{\partial^2 f_\ell}{\partial x_{3}\partial x_\ell}-\frac{1}{c_{1}(\alpha_3,x_3)}\frac{\alpha_3-1}{x_3}\frac{\partial f_3}{\partial x_3}\right)\right.\notag\\
	&\left.\qquad-\sum_{\ell=1,\,\ell\neq 3}^{3}\left(\frac{1}{c_{1}(\alpha_\ell,x_\ell)c_{1}(\alpha_3,x_3)}\frac{\partial^2 f_\ell}{\partial x_{\ell}\partial x_3}-\frac{1}{c_{1}^{2}(\alpha_\ell,x_\ell)}\left(\frac{\partial^2 f_3}{\partial x_{\ell}^2}-\frac{\alpha_\ell-1}{x_\ell}\frac{\partial f_3}{\partial x_\ell}\right)\right)+\lambda^2 f_3\right]\,\mathbf{e}_3.\notag
\end{align}

\section{Physical systems in NIDS}

We now present several examples in which the NIDS Moisil-Teodorescu operator appears in linear hydrodynamic and elastic systems. The reader can find similar examples for the case of a fractional Moisil-Teodorescu operator in \cite{BoPePe}, where the authors apply the Stillinger’s formalism.

Let $\Psi_0=\Psi_0(\vec{x})$ and $\vec{\Psi}=\vec{\Psi}(\vec{x})$ a scalar and vector field, respectively, related by
\begin{equation}\label{VF1}
\mathrm{grad}_{(\alpha)}\left[\Psi_0\right]+\mathrm{curl}_{(\alpha)} \left[\vec{\Psi}\right]+\left[\vec{b}, \vec{\Psi}\right]+\Psi_0\, \vec{a}=0,\qquad\mathrm{div}_{(\alpha)}\left[\vec{\Psi}\right]+\left\langle\vec{a},\vec{\Psi}\right\rangle=0, 
\end{equation}
where $\vec{x}$ is the position vector and $\vec{a}, \vec{b}$ are constant real-valued vectors. 

The difference of the relations in (\ref{VF1}) can be written as
\begin{equation*}
	D_{MT}^{(\alpha)}\left[\Psi_0+\vec{\Psi}\right]= \left[\vec{\Psi}, \vec{b}\right]+\left\langle\vec{a},\vec{\Psi}\right\rangle-\Psi_0\,\vec{a}.
\end{equation*}

\begin{ex}[{\bf NIDS Ideal fluid}]
	The velocity field $\vec{\Phi}$ of an ideal fluid is incompressible (solenoidal) and irrotational, i.e.,
	\begin{equation*}\label{Example2}
		\mathrm{div}_{(\alpha)}\left[\vec{\Phi}\right]=0,\qquad \mathrm{curl}_{(\alpha)}\left[\vec{\Phi}\right]=0,
	\end{equation*}
which corresponds to (\ref{VF1}) with $\Psi_0\equiv 0$, $\vec{\Psi}=\vec{\Phi}$, $\vec{a}=0$, $\vec{b}=0$.
\end{ex} 

\begin{ex}[{\bf NIDS Stokes flows}]
The velocity field $\vec{\Phi}$ (time-independent) of a viscous incompressible fluid, provided the inertial and thermal effects are negligible, is described by the equations
	\begin{equation}\label{Example3}
		\mu_0\, \vec{\Delta}^{(\alpha)}\left[\vec{\Phi}\right]=\mathrm{grad}_{(\alpha)}\left[p_0\right],\qquad\mathrm{div}_{(\alpha)}\left[\vec{\Phi}\right]=0,
	\end{equation}
where $p_0$ and $\mu_0$ are the pressure in the fluid and shear viscosity, respectively. Equations (\ref{Example3}) are best known as Stokes equations.
	
From (\ref{Example3}) one can see that the pressure $p_0$ and vorticity $\vec{\Theta}=\mathrm{curl}_{(\alpha)}\left[\vec{\Phi}\right]$ are related by
	\begin{equation*}\label{Example3_1}
		\mu_0\, \mathrm{curl}_{(\alpha)}\left[\vec{\Theta}\right]=-\mathrm{grad}_{(\alpha)}\left[p_0\right],\qquad\mathrm{div}_{(\alpha)}\left[\vec{\Theta}\right]=0, 
	\end{equation*}
that can be obtained from (\ref{VF1}) for the case $\Psi_0=p_0$, $\vec{\Psi}=\mu_0\vec{\Theta}$, $\vec{a}=0$, and $\vec{b}=0$.
\end{ex}

\begin{ex}[{\bf NIDS Lam\'{e}-Navier system}]
A displacement field $\vec{f}$ in a homogeneous isotropic linear elastic material without volume forces of the three-dimensional elasticity theory is described by the Lam\'{e}-Navier operator equation. Indeed,
	\begin{equation}\label{Lame}
		\mathcal{L}_{\gamma,\nu}\left[\vec{f}\right]:=\nu\vec{\Delta}\left[\vec{f}\right]+(\nu+\gamma)\mathrm{grad}\left[\mathrm{div}\left[\vec{f}\right]\right]=0,
	\end{equation}
where $\nu>0$, $\displaystyle\gamma>-\frac{2}{3}\nu$. For standard works along classical lines we refer the reader to \cite{La, Mi}.
	
Notice that in \cite{Mi} the operator (\ref{vector_Bitsadze_operator}) appears when the Poisson constant $\displaystyle \sigma:=\frac{\gamma}{2(\gamma+\nu)}=\frac{3}{4}$, and for this case:
	\[-\widetilde{\vec{\Delta}}^{(\alpha)}\left[\vec{f}\right]=0,\]
we have infinitely many solutions.
	
Adding (\ref{vector_conformal_Laplacian}) to (\ref{vector_Bitsadze_operator}) one gets:
	\begin{equation*}
		\mathrm{grad}_{(\alpha)}\left[\mathrm{div}_{(\alpha)}\left[\vec{f}\right]\right]=-\frac{1}{2}\left(\left(D_{MT}^{(\alpha)}\right)^2\left[\vec{f}\right]+D_{MT}^{(\alpha),r}D_{MT}^{(\alpha)}\left[\vec{f}\right]\right).
\end{equation*}
	
In fact, NIDS version of operator equation (\ref{Lame}) can be rewritten in the following way:
	\begin{equation*}
		\left(\frac{\nu+\gamma}{2}\right)D_{MT}^{(\alpha),r}D_{MT}^{(\alpha)}\left[\vec{f}\right]+\left(\nu+\frac{\nu+\gamma}{2}\right)\left(D_{MT}^{(\alpha)}\right)^2\left[\vec{f}\right]=0.
	\end{equation*}
With the notation $\displaystyle\eta:=\frac{\nu+\gamma}{2}$ and $\displaystyle\beta:=\frac{3\nu+\gamma}{2}$, we have: 
	\begin{equation*}
		\mathcal{L}_{ \gamma ,\nu}^{(\alpha)}\left[\vec{f}\right]=\eta\, D_{MT}^{(\alpha),r}D_{MT}^{(\alpha)}\left[\vec{f}\right]+\beta \left(D_{MT}^{(\alpha)}\right)^2\left[\vec{f}\right].
	\end{equation*}
The conditions relating $\gamma,\nu$ in (\ref{Lame}), imply that $\eta\neq0$ and $\beta\neq0$.
	
We may now write the Lam\'{e}-Navier operator equation in the NIDS context, using the preceding results. 
	
The corresponding expressions for $\mathcal{L}_{\gamma,\nu}^{(\alpha)}\left[\vec{f}\right]$ in Cartesian coordinates are:
	\begin{align}
		\mathcal{L}_{\gamma,\nu}^{(\alpha)}[\vec{f}]&=\nu\,\vec{\Delta}^{\alpha}\left[\vec{f}\right]+(\nu+\gamma)\,\mathrm{grad}_{(\alpha)}\left[\mathrm{div}_{(\alpha)}\left[\vec{f}\right]\right]\notag\\
		&=\nu\left(\mathrm{grad}_{(\alpha)}\left[\mathrm{div}_{(\alpha)}\left[\vec{f}\right]\right]-\mathrm{curl}_{(\alpha)}\left[\mathrm{curl}_{(\alpha)}\left[\vec{f}\right]\right]\right)\notag\\
		&\qquad+(\nu+\gamma)\,\mathrm{grad}_{(\alpha)}\left[\mathrm{div}_{(\alpha)}\left[\vec{f}\right]\right],\notag
	\end{align}
where $\mathrm{grad}_{(\alpha)}$, $\mathrm{div}_{(\alpha)}$ and $\mathrm{curl}_{(\alpha)}$ are defined in Section 2.
	
Finally, the NIDS Lam\'{e}-Navier operator equation reads:
	\begin{equation*}
		\mathcal{L}_{\gamma,\nu}^{(\alpha)}[\vec{f}]=0,
	\end{equation*}
where $\nu>0$, $\displaystyle\gamma>-\frac{2}{3}\nu$.
\end{ex}

\section{NIDS time-harmonic electromagnetic fields theory}
Let $\Omega$ be a domain in $\mathbb{R}^3$, and let $\vec{E},\vec{H}:\Omega\to\mathbb{C}^3$ be a pair of complex-valued vector fields. The following system defines the time-harmonic (monochromatic) Maxwell equations:
\begin{equation}\label{Maxwell_system}
	\begin{cases}
		\mathrm{curl}\left[\vec{H}\right]=\sigma\vec{E},\\
		\mathrm{curl}\left[\vec{E}\right]=i\omega\mu\vec{H},\\
		\mathrm{div}\left[\vec{H}\right]=0,\\
		\mathrm{div}\left[\vec{E}\right]=0,
	\end{cases}
\end{equation}
where $\sigma:=\sigma^{*}-i\omega\epsilon$ denotes the complex electrical conductivity, $\epsilon$ the dielectric constant, $\mu$  the magnetic permeability, and $\sigma^{*}$ the medium electrical conductivity, which is the inverse to its electrical resistivity $\sigma^{*}=1/\rho$. There are no currents or charges in $\Omega$ because it is supposed to be filled with a homogeneous medium. 

The pair $(\vec{E},\vec{H})$ is a time-harmonic (monochromatic) electromagnetic field when $\vec{E}$ and $\vec{H}$ constitute a solution to the system (\ref{Maxwell_system}) in $\Omega$. The electrical and magnetic components of the electromagnetic field are referred to as  $\vec{E}$ and $\vec{H}$, respectively. Moreover, the homogeneous Helmholtz equation is satisfied by both components:
\begin{align*}
	\Delta_{\mathbb{R}^3}\left[\vec{E}\right]+\lambda^2\vec{E}&=0,\\
    \Delta_{\mathbb{R}^3}\left[\vec{H}\right]+\lambda^2\vec{H}&=0,
\end{align*} 
where $\lambda^2:=i\omega\mu\sigma^{*}+\omega^2\mu\epsilon=i\omega\mu\sigma\in\mathbb{C}$ and $\lambda$ is called a medium wave number.

Consider the matrix operator \[\mathcal{M}:C^1(\Omega;\text{Mat}_{2\times2}(\mathbb{H}(\mathbb{C})))\to C^0(\Omega;\text{Mat}_{2\times2}(\mathbb{H}(\mathbb{C}))),\] where $\text{Mat}_{2\times2}(\mathbb{H}(\mathbb{C}))$ denotes the set of $2\times 2$ matrices with entries in $\mathbb{H}(\mathbb{C})$, defined as:
\begin{equation}\label{quaternionic_Maxwell_operator}
	\mathcal{M}:=
	\begin{pmatrix}
		\sigma & -D_{MT}\\
		D_{MT} & -i\omega\mu
	\end{pmatrix}.
\end{equation}

Columns $\begin{pmatrix}
	a \\
	b
\end{pmatrix}$ are naturally associated with matrices of the form $\begin{pmatrix}
	a & 0\\
	b & 0
\end{pmatrix}$.

In \cite{KravSha} the matrix operator (\ref{quaternionic_Maxwell_operator}) was called  the quaternionic Maxwell operator as restricted to the set
\[\hat{C}^1(\Omega;\mathbb{C}^3\times\mathbb{C}^3):=\left\{\begin{pmatrix}
	\vec{f}\\
	\vec{g}
\end{pmatrix}\in C^1(\Omega;\mathbb{C}^3\times\mathbb{C}^3)\,:\,\mathrm{div}\left[\vec{f}\right]=\mathrm{div}\left[\vec{g}\right]=0\right\},\]
the system (\ref{Maxwell_system}) becomes
\[\mathcal{M}\begin{pmatrix}
	\vec{E}\\
	\vec{H}
\end{pmatrix}=0.\]

We proceed to show the NIDS generalization of the quaternionic Maxwell operator on a far-field region of the NIDS. The NIDS quaternionic Maxwell operator is defined by 
\[\mathcal{M}^{(\alpha)}:C^1(\Omega;\text{Mat}_{2\times2}(\mathbb{H}(\mathbb{C})))\to C^0(\Omega;\text{Mat}_{2\times2}(\mathbb{H}(\mathbb{C}))),\] 
with
\begin{equation*}
	\mathcal{M}^{(\alpha)}:=
	\begin{pmatrix}
		\sigma & -D_{MT}^{(\alpha)}\\
		D_{MT}^{(\alpha)} & -i\omega\mu
	\end{pmatrix}.
\end{equation*}

The NIDS quaternionic Maxwell operator takes the following form:
\begin{equation*}
	\mathcal{M}^{(\alpha)}=
	\begin{pmatrix}
		\sigma & \mathrm{div}_{(\alpha)}-\mathrm{curl}_{(\alpha)}\\
		-\mathrm{div}_{(\alpha)}+\mathrm{curl}_{(\alpha)} & -i\omega\mu
	\end{pmatrix},
\end{equation*}
where $\mathrm{div}_{(\alpha)}$ and $\mathrm{curl}_{(\alpha)}$ are defined in Equations (\ref{nids_div}) and (\ref{nids_curl}), respectively.

The equation 
\[\mathcal{M}^{(\alpha)}\begin{pmatrix}
	\vec{E}\\
	\vec{H}
\end{pmatrix}=0,\]
restricted to $\hat{C}^1(\Omega;\mathbb{C}^3\times\mathbb{C}^3)$ is equivalent to
\begin{equation}\label{fractional_Maxwell_system}
	\begin{cases}
		\mathrm{curl}_{(\alpha)}\left[\vec{H}\right]=\sigma\vec{E},\\
		\mathrm{curl}_{(\alpha)}\left[\vec{E}\right]=i\omega\mu\vec{H},\\
		\mathrm{div}_{(\alpha)}\left[\vec{H}\right]=0,\\
		\mathrm{div}_{(\alpha)}\left[\vec{E}\right]=0.
	\end{cases}
\end{equation}

The system (\ref{fractional_Maxwell_system}) is a NIDS version of the conventional time-harmonic Maxwell system from integer dimensional Euclidean space.  These fractional equations in Euclidean space can be reduced to a classical time-harmonic Maxwell system for the case of all $\alpha_k=1$.
Observe that the NIDS homogeneous Helmholtz equation is satisfied by both components: 
\begin{align*}
	\vec{\Delta}^{(\alpha)}\left[ \vec{E}\right]+\lambda^2\vec{E}&=0,\\
	\vec{\Delta}^{(\alpha)}\left[ \vec{H}\right]+\lambda^2\vec{H}&=0.
\end{align*} 

\subsection{Reformulation of the time-harmonic Maxwell system in NIDS}
Let $\lambda=\omega\sqrt{\epsilon\mu}$, where the square root is selected in such a way that  $\mbox{Im}\,\lambda\geq0$. We can now establish the connection between the operators $\mathcal{M}^{(\alpha)}$ and $D_{MT}^{(\alpha)}\pm\lambda$. Indeed, let us introduce the notation for the following invertible matrices
\begin{equation*}
	A_1:=
	\begin{pmatrix}
		\lambda & -\sigma\\
		-\lambda & -\sigma
	\end{pmatrix},\,\,\,
	B_1:=\frac{1}{2}
	\begin{pmatrix}
		\sigma^{-1} & -\sigma^{-1}\\
		\lambda^{-1} & -\lambda^{-1}
	\end{pmatrix}.
\end{equation*}
A direct computation shows that
\begin{equation*}
	\mathcal{M}^{(\alpha)}_{\lambda}:=A_{1}\ast\mathcal{M}^{(\alpha)}\ast B_{1}=
	\begin{pmatrix}
		D_{MT}^{(\alpha)}-\lambda & 0\\
		0 & D_{MT}^{(\alpha)}+\lambda
	\end{pmatrix},
\end{equation*}
where “$\ast$'' stand for usual matrix multiplication.

Introduce the following pair of purely vectorial quaternionic functions 
\begin{equation}\label{BQF1}
	\vec{\varphi}:=-i\omega\epsilon\vec{E}+\lambda\vec{H}, 
\end{equation}

\begin{equation}\label{BQF2}
	\vec{\psi}:=i\omega\epsilon\vec{E}+\lambda\vec{H}.
\end{equation}
Next theorem provides a quaternionic reformulation of a NIDS time-harmonic Maxwell system.   
\begin{thm}
	The NIDS quaternionic equation
	\begin{equation}\label{equivalent_equation}
		\mathcal{M}^{(\alpha)}_{\lambda}\begin{pmatrix}
			\vec{\varphi}\\
			\vec{\psi}
		\end{pmatrix}=0,
	\end{equation}
	restricted to $\hat{C}^1(\Omega;\mathbb{C}^3\times\mathbb{C}^3)$ is equivalent to (\ref{fractional_Maxwell_system}).
\end{thm}
To be more precise, $\vec{\varphi}$ and $\vec{\psi}$ are solutions of (\ref{equivalent_equation}), if and only if $\vec{E}$ and $\vec{H}$ are solutions of (\ref{fractional_Maxwell_system}).  

\begin{proof}

Let $\vec{E}$ and $\vec{H}$ solutions of (\ref{fractional_Maxwell_system}), which may be expressed as the following quaternionic equations:
\begin{equation}\label{D1}
	D_{MT}^{(\alpha)}\left[\vec{E}\right]=i\omega\mu\vec{H},  
\end{equation}
\begin{equation}\label{D2}
	D_{MT}^{(\alpha)}\left[\vec{H}\right]=-i\omega\epsilon\vec{E}.  
\end{equation}
Applying $D_{MT}^{(\alpha)}$ to $\vec{\varphi}$ in (\ref{BQF1}) and combining (\ref{D1}) with (\ref{D2}) we get 
\begin{align}
	D_{MT}^{(\alpha)}\left[\vec{\varphi}\right]=&-i\omega\epsilon \,D_{MT}^{(\alpha)}\left[\vec{E}\right]+\lambda \,D_{MT}^{(\alpha)}\left[\vec{H}\right]\notag\\
	=&-i\omega\epsilon\left(i\omega\mu\vec{H}\right)+\lambda\left(-i\omega\epsilon\vec{E}\right)\notag\\
	=&\lambda^{2}\vec{H}-\lambda i\omega\epsilon\vec{E}\nonumber\\
	=&\lambda\left(\lambda\vec{H}-i\omega\epsilon\vec{E}\right)\notag\\
	=&\lambda\vec{\varphi}. \label{ApplyingDMTtovarphi}
\end{align}
Similarly, we can establish that $\vec{\psi}$ in (\ref{BQF2}) satisfies  
\begin{equation}\label{ApplyingDMTtopsi}
	D_{MT}^{(\alpha)}\left[\vec{\psi}\right]=-\lambda\vec{\psi}. 
\end{equation}
Thus, (\ref{ApplyingDMTtovarphi}) and (\ref{ApplyingDMTtopsi}) show that $\begin{pmatrix}
	\vec{\varphi}\\
	\vec{\psi}
\end{pmatrix}$ 
satisfies (\ref{equivalent_equation}).

Conversely, assume that  $\begin{pmatrix}
	\vec{\varphi}\\
	\vec{\psi}
\end{pmatrix}$ satisfies (\ref{equivalent_equation}).
A direct computation shows that 
\begin{equation}\label{Dphi2}
	D_{MT}^{(\alpha)}\left[\vec{\varphi}\right]=\lambda\vec{\varphi}.
\end{equation}

Substituting (\ref{BQF1}) into (\ref{Dphi2}) gives
\begin{align}
	-i\omega\epsilon \,D_{MT}^{(\alpha)}[\vec{E}]+\lambda \,D_{MT}^{(\alpha)}[\vec{H}]=&-i\omega\epsilon\lambda\vec{E}+\lambda^{2}\vec{H}\notag\\
	=&-i\omega\epsilon\lambda\vec{E}-i^2\omega^2\epsilon\mu\vec{H}\notag\\
	=&-i\omega\epsilon(i\omega\mu\vec{H})+\lambda(-i\omega\epsilon\vec{E}),\notag
\end{align}
then (\ref{D1}) and (\ref{D2}) hold. Similar considerations apply to $\vec{\psi}$. 

Finally, if one separates the scalar and vector parts in (\ref{D1})-(\ref{D2}) and uses the vector nature of $\vec{\varphi}$, $\vec{\psi}$, together with (\ref{conformalMT}), then  (\ref{fractional_Maxwell_system}) is obtained. This completes the proof.
\end{proof}

\section{Generalized Vector Operators of type $N_F^\alpha$}

\subsection{Mathematical Formulation of the Local Generalized Operator}

The vector calculus and vector differential operations in NIDS introduced in the preceding sections (based on \cite{T1} for fractal media with power-law state densities) are not isolated elements, but can be projected within a broader mathematical framework of the generalized local derivative proposed in \cite{NGLK}.

To discuss the structural equivalence, it is worth noting that in the generalized model, the action of the differential operator on a differentiable function is governed by a non-zero deterministic kernel $F(x,\alpha)$, which weights the local rate of change according to the intrinsic properties of the space or the modeled physical phenomenon. If in this general formulation we define the components of the multidimensional kernel in terms of the reciprocals of the Tarasov geometric density functions, that is:

$$F_k(x_k, \alpha_k) := c_1^{-1}(\alpha_k, x_k) = \frac{\Gamma(\alpha_k/2)}{\pi^{\alpha_k/2}} |x_k|^{1-\alpha_k}$$

for each spatial coordinate $x_k$ with its respective dimensional parameter $\alpha_k$, the NIDS Moisil-Teodorescu operator reveals as a special case of the generalized local operator $D_{(F,\alpha)}^{MT}$.

We will see that the complementarity of both approaches is fully justified, allowing us to extend the Bitsadze, Lamé-Navier, and Maxwell systems to deformed environments.

For completeness, we present below the definition of the generalized local derivative to be used.

\begin{defn}	\label{d:61} 
Given a function $f :[0,+\infty )\rightarrow \mathbb{R}$. Then the $N$-derivative of $f$ of order $\alpha $\ is defined by	
\begin{equation*}
N_{F}^{\alpha} f(t) = \lim_{h \to 0} \frac{f(t + h F(t,\alpha)) - f(t)}{h}
\end{equation*}
for all $t >0$, $\alpha \in (0,1)$ being $F(t ,\alpha )$ some absolutely continuous function.
	
Let $f$ be $N$-differentiable in some $(0,\alpha)$, and define $N_{F}^{\alpha}f (0)=\underset{t \rightarrow 0^{+}}{\lim }N_{F}^{\alpha}f (t)$, requiring that $\underset{t \rightarrow 0^{+}}{\lim }N_{F}^{\alpha}f(t )$ exists. Ordinarily differentiability of $f$ yields  	
$$N_{F}^{\alpha}f (t )=F(t,\alpha )f'(t )$$ 
where $f'(t )$ is the ordinary derivative.
\end{defn}

This definition allows for precise modulation of the differential operator and not only recovers the classical derivative when $F(t, \alpha) \equiv 1$, but also allows the integration of more complex kernels, such as the linear or Mittag-Leffler kernels. These kernels have the capacity to radically alter the physical behavior of the solutions, as we will see later, by introducing effects of acceleration, deceleration, or even wave trapping. Other necessary definitions can be obtained from \cite{LNV} that will be used in the future.

Let $\Omega \subset \mathbb{R}^3$ be an open domain. We define the components of the local generalized partial derivative with respect to the $k$-th coordinate analogously to the property of the fundamental theorem:
$$\partial_{x_k, F}^{\alpha_k} f(\vec{x}) := F_k(x_k, \alpha_k) \frac{\partial f}{\partial x_k}(\vec{x}), \quad k=1,2,3$$
where $F_k(x_k, \alpha_k)$ is a positive continuous local kernel (it can be indexed in the Mittag-Leffler function). The generalized nabla operator is denoted as:
$$\nabla_{(F, \alpha)} = \sum_{k=1}^3 e_k F_k(x_k, \alpha_k) \frac{\partial}{\partial x_k} $$
From this operator, the generalized local gradient, divergence, and curl are naturally defined for a scalar field $f$ and a vector field $\vec{A} = \sum_{k=1}^3 A_k e_k$.

\begin{thm}\label{t:61}
(The Local Generalized Moisil-Teodorescu Operator $D_{MT}^{(F,\alpha)}$)

Let $f = f_0 + \vec{f}$ be a function with quaternionic values in $\Omega$, where $f_0$ is the scalar component and $\vec{f} = \sum _{k=1}^3 f_k e_k$ is the vector component. The Local Generalized Moisil-Teodorescu Operator is defined by:

$$D_{MT}^{(F,\alpha)}[f] := - \operatorname{div}_{(F,\alpha)}[\vec{f}] + \operatorname{grad}_{(F,\alpha)}[f_0] + \operatorname{curl}_{(F,\alpha)}[\vec{f}]$$

If the components of intervening kernels satisfy the cross-symmetry condition $\frac{\partial}{\partial x_j}[F_k(x_k, \alpha_k)] = 0$ for all $j \neq k$, then the operator $D_{MT}^{(F,\alpha)}$ structurally preserves the curl-nullity property of the gradient:
$$\operatorname{curl}_{(F,\alpha)}[\operatorname{grad}_{(F,\alpha)}[f_0]] = 0$$
\end{thm}
\begin{proof}
By definition of the generalized rotational operator in Cartesian coordinates distorted by the kernels of type $N_F^\alpha$, the $i$-th component of $\operatorname{curl}_{(F,\alpha)}[\operatorname{grad}_{(F,\alpha)}[f_0]]$ is given using the Levi-Civita symbol $\epsilon_{ijk}$:
$$\left(\operatorname{curl}_{(F,\alpha)}[\operatorname{grad}_{(F,\alpha)}[f_0]]\right)_i = \sum _{j=1}^3 \sum _{k=1}^3 \epsilon _{ijk} F_j(x_j, \alpha_j) \frac{\partial}{\partial x_j} \left( F_k(x_k, \alpha_k) \frac{\partial f_0}{\partial x_k} \right).$$
Applying the ordinary product rule for partial derivatives in Euclidean space:
$$\frac{\partial}{\partial x_j} \left( F_k(x_k, \alpha_k) \frac{\partial f_0}{\partial x_k} \right) = \frac{\partial F_k(x_k, \alpha_k)}{\partial x_j} \frac{\partial f_0}{\partial x_k} + F_k(x_k, \alpha_k) \frac{\partial^2 f_0}{\partial x_j \partial x_k}.$$
Due to the hypothesis that $F_k$ depends only on its own coordinate variable $x_k$ (one-dimensional symmetry of the fractal medium along axes), we have that for all $j \neq k$, $\frac{\partial F_k(x_k, \alpha_k)}{\partial x_j} = 0$. Therefore, the kernel derivative term vanishes in all summands where $\epsilon_{ijk} \neq 0$ (since the non-zero components of the tensor require $j \neq k$).

Substituting the remaining term into the double summation:
$$\left(\operatorname{curl}_{(F,\alpha)}[\operatorname{grad}_{(F,\alpha)}[f_0]]\right)_i = \sum_{j,k=1}^3 \epsilon_{ijk} F_j(x_j, \alpha_j) F_k(x_k, \alpha_k) \frac{\partial^2 f_0}{\partial x_j \partial x_k}.$$
The structural coefficient tensor $F_j(x_j, \alpha_j) F_k(x_k, \alpha_k)$ is symmetric with respect to the interchange of indices $j$ and $k$, and by the Clairaut-Schwarz Theorem, the mixed partial derivatives are also symmetric ($\displaystyle\frac{\partial^2 f_0}{\partial x_j \partial x_k} = \displaystyle\frac{\partial^2 f_0}{\partial x_k \partial x_j}$). Since the Levi-Civita symbol $\epsilon_{ijk}$ is completely antisymmetric with respect to $j$ and $k$ ($\epsilon_{ijk} = -\epsilon_{ikj}$), the contraction of a symmetric tensor with an antisymmetric tensor cancels out identically:
$$\left(\operatorname{curl}_{(F,\alpha)}[\operatorname{grad}_{(F,\alpha)}[f_0]]\right)_i = 0, \quad \forall i=1,2,3$$
Which formally proves that $\operatorname{curl}_{(F,\alpha)}[\operatorname{grad}_{(F,\alpha)}[f_0]] = 0$.
\end{proof}

\begin{thm}\label{t:62}(Factorization and Perturbation of the Generalized Quaternion Laplacian)

The square of the local generalized Moisil-Teodorescu operator of type $N_F^\alpha$ decomposes the quaternion space and induces a nonhomogeneous generalized Laplacian operator of the form:

$$(D_{MT}^{(F,\alpha)})^2 [f] = -\Delta_{\mathbb{H}}^{(F,\alpha)}[f] - \mathcal{R}^{(F,\alpha)}[f],$$

where $\Delta_{\mathbb{H}}^{(F,\alpha)}$ is the pure Laplacian scaled by the squared kernels:

$$\Delta_{\mathbb{H}}^{(F,\alpha)}[f] = \sum_{k=1}^3 F_k^2(x_k, \alpha_k) \frac{\partial^2 f}{\partial x_k^2}$$

and $\mathcal{R}^{(F,\alpha)}[f]$ is an integer-order convective geometric perturbation term generated by the intrinsic variation of the kernel:

$$\mathcal{R}^{(F,\alpha)}[f] = \sum _{k=1}^3 F_k(x_k, \alpha_k) \left( \frac{\partial F_k(x_k, \alpha_k)}{\partial x_k} \right) \frac{\partial f}{\partial x_k}.$$
\end{thm}
\begin{proof}
Starting from the intrinsic definition of the quaternion product and considering the algebraic operator $D_{MT}^{(F,\alpha)} = \sum_{k=1}^3 e_k F_k(x_k, \alpha_k) \frac{\partial}{\partial x_k}$, we calculate its successive mapping on a sufficiently smooth quaternion function $f$:
$$(D_{MT}^{(F,\alpha)})^2 [f] = \left( \sum_{j=1}^3 e_j F_j(x_j, \alpha_j) \frac{\partial}{\partial x_j} \right) \left( \sum_{k=1}^3 e_k F_k(x_k, \alpha_k) \frac{\partial f}{\partial x_k} \right).$$
Expanding the double sum by linearity:
$$(D_{MT}^{(F,\alpha)})^2 [f] = \sum_{j=1}^3 \sum_{k=1}^3 e_j e_k F_j(x_j, \alpha_j) \frac{\partial}{\partial x_j} \left( F_k(x_k, \alpha_k) \frac{\partial f}{\partial x_k} \right)$$
Performing ordinary differentiation under the operator $\displaystyle\frac{\partial}{\partial x_j}:$
$$(D_{MT}^{(F,\alpha)})^2 [f] = \sum _{j=1}^3 \sum _{k=1}^3 e_j e_k F_j(x_j, \alpha_j) \left[ \frac{\partial F_k(x_k, \alpha_k)}{\partial x_j} \frac{\partial f}{\partial x_k} + F_k(x_k, \alpha_k) \frac{\partial^2 f}{\partial x_j \partial x_k} \right]$$
Separating the summation into diagonal terms ($j=k$) and mixed terms ($j \neq k$).

\textbf{Case 1: Diagonal terms ($j=k$)}
Taking into account the fundamental properties of quaternionic imaginary units ($e_k^2 = -1$):
\begin{eqnarray*}
&&\sum_{k=1}^3 e_k^2 F_k(x_k, \alpha_k) \left[ \frac{\partial F_k(x_k, \alpha_k)}{\partial x_k} \frac{\partial f}{\partial x_k} + F_k(x_k, \alpha_k) \frac{\partial^2 f}{\partial x_k^2} \right]\\
&&= -\sum_{k=1}^3 F_k^2(x_k, \alpha_k) \frac{\partial^2 f}{\partial x_k^2} - \sum_{k=1}^3 F_k(x_k, \alpha_k) \frac{\partial F_k(x_k, \alpha_k)}{\partial x_k} \frac{\partial f}{\partial x_k}.
\end{eqnarray*}

\textbf{Case 2: Mixed Terms ($j \neq k$)}
For $j \neq k$, due to the independence of the kernels with respect to the variables, $\frac{\partial F_k(x_k, \alpha_k)}{\partial x_j} = 0$. Therefore, the mixed expression reduces to:
$$\sum_{j \neq k} e_j e_k F_j(x_j, \alpha_j) F_k(x_k, \alpha_k) \frac{\partial^2 f}{\partial x_j \partial x_k}.$$
Grouping the symmetric pairs $(j,k)$ and $(k,j)$ and using the quaternionic anticommutativity property $e_j e_k = -e_k e_j$ we have:
$$e_j e_k F_j F_k \frac{\partial^2 f}{\partial x_j \partial x_k} + e_k e_j F_k F_j \frac{\partial^2 f}{\partial x_k \partial x_j} = (e_j e_k + e_k e_j) F_j F_k \frac{\partial^2 f}{\partial x_j \partial x_k} = 0.$$
Consequently, the entire sum of mixed terms vanishes identically.

Combining the results of Case 1, we obtain:
$$(D_{MT}^{(F,\alpha)})^2 [f] = -\left( \sum _{k=1}^3 F_k^2(x_k, \alpha_k) \frac{\partial^2 f}{\partial x_k^2} \right) - \sum _{k=1}^3 F_k(x_k, \alpha_k) \frac{\partial F_k(x_k, \alpha_k)}{\partial x_k} \frac{\partial f}{\partial x_k}.$$
Defining $\Delta_{\mathbb{H}}^{(F,\alpha)}[f] := \sum_{k=1}^3 F_k^2(x_k, \alpha_k) \frac{\partial^2 f}{\partial x_k^2}$ and algebraically assigning the sign to the residual operator $\mathcal{R}^{(F,\alpha)}[f] := - \sum_{k=1}^3 F_k(x_k, \alpha_k) \frac{\partial F_k(x_k, \alpha_k)}{\partial x_k} \frac{\partial f}{\partial x_k}$, we arrive at the required perturbed factorization equation.
\end{proof}

\begin{rem}
In the above parameterization, this residue necessarily takes the form of a rigid power $\frac{\alpha_k-1}{x_k}$. When using the generalized derivative, the structure of the residue depends directly on the behavior of the specific kernel, sweeping linear, hyperbolic, or asymptotic dynamics optimal for porous or fractal continuous media.
\end{rem}

\subsection{Energy Estimates and Physical Foundations of the Residue Operator}

To address the physical consistency of the additional term $\mathcal{R}^{(F,\alpha)}[f]$, we now present explicit relationships, energy estimates, and dimensional analysis of that term.

\begin{defn}	\label{d:620} 

In a deformed medium, characterized by the kernel $F_k(x_k, \alpha_k)$, the linear strain tensor $\boldsymbol{\varepsilon}^{(F,\alpha)}$ associated with a displacement field $\vec{u}$ is defined as:

$$\boldsymbol{\varepsilon}^{(F,\alpha)} := \frac{1}{2} \left( \nabla_{(F,\alpha)} \vec{u} + (\nabla_{(F,\alpha)} \vec{u})^T \right)$$

The corresponding Cauchy strain tensor follows the isotropic Hooke-Tarasov constitutive law:

$$\boldsymbol{\sigma}^{(F,\alpha)} = \lambda \operatorname{tr}(\boldsymbol{\varepsilon}^{(F,\alpha)}) \mathbf{I} + 2\mu \boldsymbol{\varepsilon}^{(F,\alpha)}$$

where $\lambda, \mu$ are the Lamé constants. The invariant volume element of the metric space is $d\mu_{(F,\alpha)} = \prod_{k=1}^3 \frac{dx_k}{F_k(x_k, \alpha_k)}$.
\end{defn}

\begin{thm}\label{t:621}(Energy Balance and Geometric Dissipation in Elasticity)

Let $\vec{u}(\vec{x},t)$ satisfy the dynamic elastodynamic equation in deformed space:

$$\rho \frac{\partial^2 \vec{u}}{\partial t^2} = (\lambda + 2\mu) \operatorname{grad}_{(F,\alpha)}(\operatorname{div}_{(F,\alpha)} \vec{u}) - \mu \operatorname{curl}_{(F,\alpha)}(\operatorname{curl}_{(F,\alpha)} \vec{u})$$

The total mechanical energy is given by:

$$E(t) = \int_{\Omega} \left( \frac{1}{2} \rho \left\vert{}\frac{\partial \vec{u}}{\partial t}\right\vert{}^2 + \frac{1}{2} \boldsymbol{\sigma}^{(F,\alpha)} : \boldsymbol{\varepsilon}^{(F,\alpha)} \right) d\mu_{(F,\alpha)}$$

satisfies the energy balance equation:

$$\frac{d E(t)}{d t} = -\int_{\Omega} \left\langle \frac{\partial \vec{u}}{\partial t}, \, \mathcal{R}^{(F,\alpha)}[\vec{u}] \right\rangle d\mu_{(F,\alpha)}$$
\end{thm}
\begin{proof}

Taking the inner product of the equation of motion with $\frac{\partial \vec{u}}{\partial t}$ and integrating over $\Omega$ with respect to $d\mu_{(F,\alpha)}$:

$$\int_{\Omega} \rho \left\langle \frac{\partial^2 \vec{u}}{\partial t^2}, \frac{\partial \vec{u}}{\partial t} \right\rangle d\mu_{(F,\alpha)} = \frac{d}{dt} \int_{\Omega} \frac{1}{2} \rho \left\vert{} \frac{\partial \vec{u}}{\partial t} \right\vert{}^2 d\mu_{(F,\alpha)}$$

Expanding the vector operators using Theorem \ref{t:62}, the elasticity operator is separated into the standard variation of the strain energy plus the first-order convective term $-\mathcal{R}^{(F,\alpha)}[\vec{u}]$. Integration by parts under homogeneous boundary conditions $\vec{u}\vert{}_{\partial\Omega} = 0$ leads to:

$$\frac{d E(t)}{d t} = -\int_{\Omega} \sum_{k=1}^3 F_k(x_k,\alpha_k) \left( \frac{\partial F_k(x_k,\alpha_k)}{\partial x_k} \right) \left\langle \frac{\partial \vec{u}}{\partial x_k}, \frac{\partial \vec{u}}{\partial t} \right\rangle d\mu_{(F,\alpha)}$$
\end{proof}

\begin{rem}(Physical Interpretation and Dimensional Analysis)

If $\frac{\partial F_k}{\partial x_k} > 0$, the power dissipated by $\mathcal{R}^{(F,\alpha)}[\vec{u}]$ is strictly negative ($\frac{dE}{dt} \leq 0$). Thus, $\mathcal{R}^{(F,\alpha)}[\vec{u}]$ acts as an intrinsic geometric viscous friction (spatial damping), transferring kinetic energy in metric microfluctuations. From a dimensional perspective, in the Tarasov model $F_k(x_k, \alpha_k) = \frac{\Gamma(\alpha_k/2)}{\pi^{\alpha_k/2}} \vert{}x_k\vert{}^{1-\alpha_k}$, we have $[F_k] = L^{1-\alpha_k}$ and $[\frac{\partial F_k}{\partial x_k}] = L^{-\alpha_k}$. Consequently, $[\Delta_{\mathbb{H}}^{(F,\alpha)}] = L^{-2}$ (second-order metric scaling) and $[\mathcal{R}^{(F,\alpha)}] = L^{-1}$ (first-order convective transport), which confirms dimensional homogeneity in non-integer dimensions.
\end{rem}

\begin{thm}\label{t:622}(Modified Poynting Theorem and Geometric Conductivity)

In temporal harmonic electrodynamics with $\operatorname{curl}_{(F,\alpha)} \vec{H} = -i\omega\epsilon \vec{E} + \vec{J}_{ext}$ and $\operatorname{curl}_{(F,\alpha)} \vec{E} = i\omega\mu \vec{H}$, the Poynting flux $\vec{S} = \frac{1}{2} \operatorname{Re}(\vec{E} \times \vec{H}^*)$ satisfies the equality:

$$\operatorname{div}_{(F,\alpha)} \vec{S} + \sigma _{geom} \vert{}\vec{E}\vert{}^2 = -\frac{1}{2} \operatorname{Re}(\vec{J}_{ext}^* \cdot \vec{E})$$

where the induced geometric conductivity $\sigma_{geom}(\vec{x})$ is given by:

$$\sigma_{geom}(\vec{x}) = \frac{1}{\omega\mu} \sum_{k=1}^3 F_k(x_k, \alpha_k) \frac{\partial F_k(x_k, \alpha_k)}{\partial x_k}$$
\end{thm}
\begin{proof}

Calculation and substituting the field equations we obtain $i\omega\mu \vert{}\vec{H}\vert{}^2 - i\omega\epsilon \vert{}\vec{E}\vert{}^2 - \langle \vec{E}, \vec{J}_{ext}^* \rangle$.

Taking the real part and applying Theorem \ref{t:62} to the decoupled wave equation, we obtain that $\mathcal{R}^{(F,\alpha)}[\vec{E}]$ is an ohmic loss term $\sigma_{geom} \vert{}\vec{E}\vert{}^2$.
\end{proof}

\subsection{Applications to Deformed Physical Systems}

\begin{thm}\label{t:63}(Reformulation of the Generalized Harmonic Maxwell System)

Let $\Omega \subset \mathbb{R}^3$ be an open domain. Consider the complex electromagnetic vector fields $\vec{E}, \vec{H} \in C^1(\Omega; \mathbb{C}^3)$ that satisfy the zero divergence conditions $div_{(\alpha)}[\vec{E}] = 0$ and $div_{(\alpha)}[\vec{H}] = 0$.

Let us define the quaternionic vector fields $\vec{\varphi}, \vec{\psi}: \Omega \rightarrow \mathbb{H}(\mathbb{C})$ by:

$$\vec{\varphi} := -i\omega\epsilon\vec{E} + \lambda\vec{H} \quad \text{y} \quad \vec{\psi} := i\omega\epsilon\vec{E} + \lambda\vec{H} \quad \text{where } \lambda = \omega\sqrt{\epsilon\mu} \text{ e } Im(\lambda) \ge 0$$

Then, the pair $(\vec{E}, \vec{H})$ satisfies the time-harmonic Maxwell-NIDS system:

$$\begin{cases} curl_{(\alpha)}[\vec{H}] = \sigma\vec{E} \\ curl_{(\alpha)}[\vec{E}] = i\omega\mu\vec{H} \\ div_{(\alpha)}[\vec{H}] = 0 \\ div_{(\alpha)}[\vec{E}] = 0 \end{cases}$$

if and only if the pair $(\vec{\varphi}, \vec{\psi})$ satisfies the matrix quaternionic equation:

$$\mathcal{M}_{\lambda}^{(\alpha)} \begin{pmatrix} \vec{\varphi} \\ \vec{\psi} \end{pmatrix} = \begin{pmatrix} (D_{MT}^{(\alpha)} - \lambda\mathcal{I})[\vec{\varphi}] \\ (D_{MT}^{(\alpha)} + \lambda\mathcal{I})[\vec{\psi}] \end{pmatrix} = \begin{pmatrix} 0 \\ 0 \end{pmatrix}$$
\end{thm}
\begin{proof}

By definition, the action of the operator $D_{MT}^{(\alpha)}$ on a pure vector field $\vec{f}$ (where the scalar part $f_0 = 0$) is given by:

\begin{equation}\label{e:631}
D_{MT}^{(\alpha)}[\vec{f}] = -div_{(\alpha)}[\vec{f}] + curl_{(\alpha)}[\vec{f}]
\end{equation}

Since $div_{(\alpha)}[\vec{E}] = 0$ and $div_{(\alpha)}[\vec{H}] = 0$, the operator reduces to:

$$D_{MT}^{(\alpha)}[\vec{E}] = curl_{(\alpha)}[\vec{E}] \quad \text{y} \quad D_{MT}^{(\alpha)}[\vec{H}] = curl_{(\alpha)}[\vec{H}]$$

Consequently, the rotational equations of the Maxwell NIDS system are equivalent to:

$$D_{MT}^{(\alpha)}[\vec{E}] = i\omega\mu\vec{H} \quad \text{y} \quad D_{MT}^{(\alpha)}[\vec{H}] = \sigma\vec{E} = -i\omega\epsilon\vec{E} \quad (\text{since } \sigma = -i\omega\epsilon)$$

Assuming that $\vec{E}$ and $\vec{H}$ satisfy the Maxwell NIDS system.

Applying the linearity of the operator $D_{MT}^{(\alpha)}$ on the field $\vec{\varphi}$:

$$D_{MT}^{(\alpha)}[\vec{\varphi}] = D_{MT}^{(\alpha)}[-i\omega\epsilon\vec{E} + \lambda\vec{H}] = -i\omega\epsilon D_{MT}^{(\alpha)}[\vec{E}] + \lambda D_{MT}^{(\alpha)}[\vec{H}]$$

Substituting now into the equation \eqref{e:631}:

$$D_{MT}^{(\alpha)}[\vec{\varphi}]=-i\omega\epsilon(i\omega\mu\vec{H}) + \lambda(-i\omega\epsilon\vec{E})= \omega^2\epsilon\mu\vec{H} - i\omega\epsilon\lambda\vec{E}$$

Taking into account that $\lambda^2 = \omega^2\epsilon\mu$:

$$D_{MT}^{(\alpha)}[\vec{\varphi}] = \lambda^2\vec{H} - i\omega\epsilon\lambda\vec{E} = \lambda(-i\omega\epsilon\vec{E} + \lambda\vec{H}) = \lambda\vec{\varphi}$$

By rearranging terms we obtain the first component of the matrix:

$$(D_{MT}^{(\alpha)} - \lambda\mathcal{I})[\vec{\varphi}] = 0$$

An analogous procedure on the field $\vec{\psi}$ results in:

$$D_{MT}^{(\alpha)}[\vec{\psi}] = D_{MT}^{(\alpha)}[i\omega\epsilon\vec{E} + \lambda\vec{H}]=i\omega\epsilon(i\omega\mu\vec{H}) + \lambda(-i\omega\epsilon\vec{E})= -\lambda^2\vec{H} - i\omega\epsilon\lambda\vec{E}=-\lambda(i\omega\epsilon\vec{E} + \lambda\vec{H})=-\lambda\vec{\psi}$$

Obtaining the second component:

$$(D_{MT}^{(\alpha)} + \lambda\mathcal{I})[\vec{\psi}] = 0$$

Suppose $D_{MT}^{(\alpha)}[\vec{\varphi}] = \lambda\vec{\varphi}$. Substituting the explicit definition of $\vec{\varphi}$ we obtain:

$$-i\omega\epsilon D_{MT}^{(\alpha)}[\vec{E}] + \lambda D_{MT}^{(\alpha)}[\vec{H}] = -i\omega\epsilon\lambda\vec{E} + \lambda^2\vec{H}$$

Substituting $\lambda^2 = \omega^2\epsilon\mu$ on the right side:

$$-i\omega\epsilon D_{MT}^{(\alpha)}[\vec{E}] + \lambda D_{MT}^{(\alpha)}[\vec{H}] = -i\omega\epsilon\lambda\vec{E} + \omega^2\epsilon\mu\vec{H} = -i\omega\epsilon(\lambda\vec{E} + i\omega\mu\vec{H})$$

Equating the real and imaginary parts or the coefficients associated with the physical parameters $(\epsilon, \mu)$, we deduce:

$$D_{MT}^{(\alpha)}[\vec{E}] = i\omega\mu\vec{H} \quad \text{y} \quad D_{MT}^{(\alpha)}[\vec{H}] = -i\omega\epsilon\vec{E}$$

Finally, by taking the scalar part and the vector part of the quaternionic relation through the decomposition $D_{MT}^{(\alpha)}[\vec{f}] = -div_{(\alpha)}[\vec{f}] + curl_{(\alpha)}[\vec{f}]$, the annulment of the scalar part guarantees $div_{(\alpha)}[\vec{E}] = 0$ and $div_{(\alpha)}[\vec{H}] = 0$, while the vector part recovers Maxwell's equations for curls.
\end{proof}

\begin{ex}
To illustrate clearly and visually how the generalized local derivative framework contains Tarasov's Non-Integer Dimension Spaces (NIDS) approach as a special case, we can construct a numerical example based on a first-order linear differential equation (a decay model).

Consider the generalized differential equation for a state function $y(x)$:
$$D_{(F,\alpha)} y(x) = - \lambda y(x), \quad y(1) = y_0,$$
where $D_{(F,\alpha)} y(x) = F(x, \alpha) \displaystyle\frac{dy}{dx}$. The general analytical solution, assuming a lower limit at $x=1$ to avoid singularities at the origin, is given by:
$$y(x) = y_0 \exp\left( -\lambda \int_{1}^{x} \frac{1}{F(t, \alpha)} dt \right)$$

\textbf{1. Tarasov Case (NIDS)}

In the NIDS approach, the kernel $F(x, \alpha)$ is rigidly coupled to the fractal geometry of the space through the inverse density of states:
$$F_{\text{Tarasov}}(x, \alpha) = c_1^{-1}(\alpha, x) = \frac{\Gamma(\alpha/2)}{\pi^{\alpha/2}} x^{1-\alpha}.$$
Substituting this into the solution, the integral results in a pure power law:
$$y_{\text{Tarasov}}(x) = y_0 \exp\left( -\lambda \frac{\pi^{\alpha/2}}{\Gamma(\alpha/2)} \frac{x^\alpha - 1}{\alpha} \right)$$

\textbf{2. Generalized Case}

The DLG framework allows us to define more flexible kernels $F(x, \alpha)$ that recover complex physical behaviors (such as saturation effects, logarithmic scales, or phase transitions in porous media) that the rigid NIDS model cannot describe on its own.

For a motivating example, consider a generalized kernel that models a medium where the effective dimension changes or stabilizes logistically on larger scales:
$$F_{\text{Generalized}}(x, \alpha) = \frac{\Gamma(\alpha/2)}{\pi^{\alpha/2}} x^{1-\alpha} \cdot \left( 1 + \beta \ln(x) \right).$$
When $\beta = 0$, we recover exactly the Tarasov case. When $\beta > 0$, the kernel introduces a structural correction to transport in non-integer metric space. The analytical solution is:
$$y_{\text{Generalized}}(x) = y_0 \exp\left( -\lambda \frac{\pi^{\alpha/2}}{\Gamma(\alpha/2)} \frac{x^\alpha (1 + \beta \ln(x)) - 1 - \frac{\beta}{\alpha}(x^\alpha - 1)}{\alpha} \right).$$
Set the following values for the numerical simulation:

$y_0 = 10$ (Initial condition at $x=1$)

$\lambda = 0.5$ (Decay constant)

$\alpha = 0.7$ (Fractional/non-integer dimension of the space)

$\beta = 0.4$ (Generalized deviation parameter)

\begin{figure}[h] 
    \centering
    \includegraphics[width=0.9\textwidth]{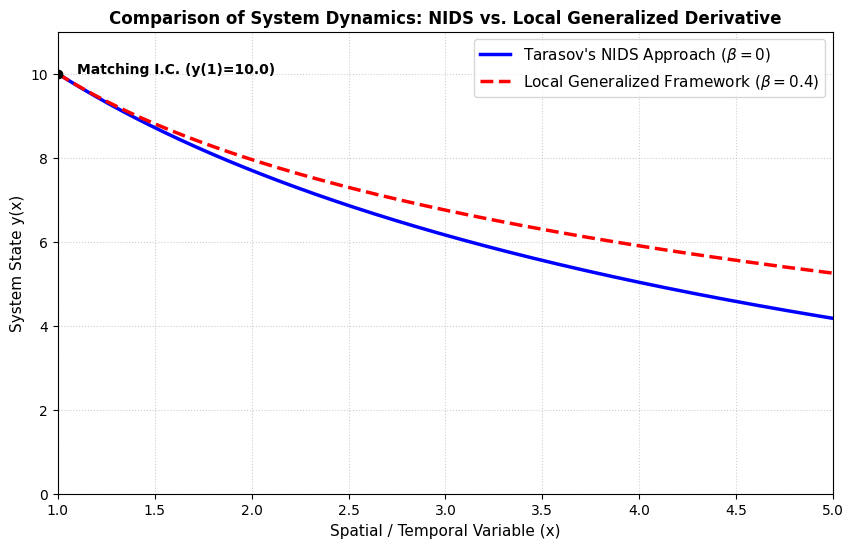}
    \caption{Comparison between the generalized model and the NIDS}
    \label{figure}
\end{figure}

The figure shows the following characteristics of the curves:
\begin{itemize}
\item Common Starting Point: Both curves originate precisely at $(1, 10)$, confirming that they share the same physical space and the same initial boundary condition.
\item Trajectory Divergence: As $x > 1$, the red curve (Generalized) decays at a different rate than the blue curve (Tarasov). This is because the kernel $F_{\text{Generalized}}$ includes a variable logarithmic damping term $\left(1 + \beta \ln(x)\right)$.
\item Proof of Generality: Modifying the code and assign $\beta = 0$, we will see that the red curve overlaps identically with the blue Tarasov curve. This numerically shows the central argument of coupling section: the DLG formalism does not contradict Tarasov, but absorbs it as a stationary or ideal state, allowing to grade the deformation of space by freely choose of the kernel profile. The separation of the two curves as $x \to 5$ shows how the generalized kernel responds to phenomena where the density of states (DOS) is not a rigid power law, but changes dynamically with scale (logarithmic scaling), capturing dynamics that Tarasov'approach cannot compactly be described in a pure non-integer space.
\end{itemize}
\end{ex}

\subsection{Physical Interpretation of the Residual Term $R_{(F,\alpha)}[f]$ in Deformed Media.}
To understand the phenomenological impact of the generalized local derivative on developed physical models, it is imperative to analyze the nature of the residual operator $R_{(F,\alpha)}[f]$ obtained in Theorem \ref{t:63}. From an analytical perspective, the modified wave or diffusion differential equation in this generalized space not only alters the metric of the classical Laplacian but also introduces a first-order perturbation coupled to the geometry of the medium.\\

\textit{1. In the Theory of Three-Dimensional Elasticity (Lamé-Navier Equation)}

In the classical Lamé-Navier model, the dynamic equilibrium of a homogeneous and isotropic medium is governed by elastic restoring forces proportional to the second derivatives of the displacement (the stress tensor). Projecting the equations into the framework of the generalized local derivative, the appearance of the term:
$$R_{(F,\alpha)}[\vec{u}] = -\sum_{k=1}^3 F_k(x_k, \alpha_k) \frac{\partial F_k(x_k, \alpha_k)}{\partial x_k} \frac{\partial \vec{u}}{\partial x_k}$$
substantially modifies the equation of motion. Physically, this term, proportional to the first spatial derivative of the displacement $\frac{\partial \vec{u}}{\partial x_k}$, acts as an effective viscous friction force or spatial geometric damping. Unlike ordinary kinematic viscosity (which dissipates energy temporarily through a term $\displaystyle\frac{\partial \vec{u}}{\partial t}$), this is an intrinsic geometric dissipation. This means that as the elastic wave propagates through the fractal medium, it experiences an apparent energy loss (attenuation) due solely to the lack of homogeneity in the density of states (DOS) of the metric space. The spatial variations of the kernel, represented by the gradient of the metric profile $\displaystyle\frac{\partial F_k}{\partial x_k}$, act as micro-obstacles or confinement fluctuations that disperse the elastic wave's energy without the need to incorporate a macroscopic thermal friction mechanism.\\

\textit{2. In Fractal Electrodynamics (Time-Harmonic Maxwell System)}

An analogous and extremely rich interpretation occurs in the modified Maxwell system. In a homogeneous ideal medium, the wave equations for the electric field $\vec{E}$ and the magnetic field $\vec{H}$ conserve the energy of the electromagnetic wavefront. However, when operating under the formalism associated to $D_{(F,\alpha)}^{MT}$, the square structure of the Moisil-Teodorescu operator pulls the residue $R_{(F,\alpha)}$ towards the generalized Helmholtz equations:
$$\Delta_{(F,\alpha)\mathbb{H}}[\vec{E}] - R_{(F,\alpha)}[\vec{E}] + \omega^2\mu\varepsilon\vec{E} = 0.$$
In this context, the operator $R_{(F,\alpha)}[\vec{E}]$ behaves mathematically similarly to the ohmic loss term $\sigma \displaystyle\frac{\partial \vec{E}}{\partial t}$, in a conductor with finite conductivity $\sigma$. Therefore, the gradient of the local kernel $\displaystyle\frac{\partial F_k}{\partial x_k}$ induces a virtual geometric conductivity. Electromagnetic waves traveling through this non-integer dimension space undergo amplitude attenuation and phase dispersion. This phenomenon is characteristic of propagation in disordered, porous, or fractal media (such as rough dielectrics or metamaterials), where the electromagnetic density of states (DOS) varies locally. The residual operator thus demonstrates that the metric distortion of space is physically equivalent to the presence of a dissipative or absorbent medium, directly connecting non-integer fractal topology with the thermodynamics of propagation processes.

\begin{rem}

It is important to note that the first-order residual operator $\mathcal{R}^{(F,\alpha)}[f]$, derived from the factorization of the generalized Laplacian, should not be interpreted strictly as an intrinsic physical mechanism (such as standard kinematic viscosity or the actual conductivity of the material). Rather, it represents a formal geometric perturbation induced by the spatial inhomogeneity of the density of fractal states (DOS) within the non-integer dimension framework. The energetic and dissipative analogies discussed in this section serve as a heuristic representation of the continuum, modeling how local geometric distortions emulate damping or attenuation phenomena in continuous media.
\end{rem}

\section{Conclusions}

In this work, we integrate calculus in non-integer dimensional spaces with quaternionic analysis; the main results obtained can be summarized as follows:

\begin{itemize}

\item The Non-Integer Dimensional Space (NIDS) version of the quaternionic Moisil-Teodorescu operator ($D_{MT}^{(\alpha)}$) and its right-sided counterpart ($D_{MT}^{(\alpha),r}$) were formulated based on vectorial differential operators scaled by Tarasov's density of states (DOS) functions.

\item The factorization of the NIDS quaternionic Laplacian ($\Delta_{\mathbb{H}}^{(\alpha)}$) via the square of the $D_{MT}^{(\alpha)}$ operator was demonstrated, as was the derivation of the NIDS Bitsadze operator ($\tilde{\Delta}_{\mathbb{H}}^{(\alpha)}$) through the composition $D_{MT}^{(\alpha),r}D_{MT}^{(\alpha)}$.

\item The Lamé-Navier system equation in continuous fractal media was reformulated, expressing it in terms of algebraic combinations of the quadratic factorizations of $D_{MT}^{(\alpha)}$ and compositions involving the right-sided operator.

\item An exact equivalence theorem was established between the time-harmonic Maxwell system of equations under the NIDS metric and a decoupled quaternionic matrix equation $(D_{MT}^{(\alpha)} \mp \lambda\mathcal{I})$ acting on the hypercomplex fields $\vec{\varphi}$ and $\vec{\psi}$.

\item The formulation was extended to the framework of generalized local derivatives (the $D_{(F,\alpha)}^{MT}$ operator for a kernel $F_k$), analytically demonstrating that squaring the operator yields a decomposition consisting of a scaled pure Laplacian ($\Delta_{\mathbb{H}}^{(F,\alpha)}$) and a first-order residual operator ($\mathcal{R}^{(F,\alpha)}$). The function of the residual term $\mathcal{R}^{(F,\alpha)}$ was quantified:

In elasticity, it represents spatial damping or intrinsic geometric friction that dissipates the wave's mechanical energy ($\frac{dE}{dt} \le 0$) due to kernel gradients.

In electrodynamics, it induces a virtual geometric conductivity ($\sigma_{\text{geom}}$) equivalent to ohmic losses or spatial dispersion caused by the heterogeneities of the continuous medium.

\end{itemize}

\section*{Statements and Declarations}
\subsection*{Funding} The first and third author were partially supported by Instituto Polit\'ecnico Nacional in the framework of SIP programs (grant numbers IND-2026-0101, SIP20260491). The second author was partially supported by Fundaci\'{o}n Universidad de las Am\'{e}ricas Puebla.
\subsection*{Competing Interests} The authors declare that they have no competing interests that could have appeared to influence the work reported in this paper.
\subsection*{Author contributions} JBR and MAPR conceived the study and wrote of first draft of the paper. JOGC and JENV oversaw the project progress. All authors provided critical feedback and helped shape the research, all  contributed to the final version of the manuscript. 
\subsection*{ORCID}
\noindent
Juan Bory-Reyes: https://orcid.org/0000-0002-7004-1794
\\
Marco Antonio P\'erez-de la Rosa:  https://orcid.org/0000-0003-3498-2914
\\
Jos\'e Oscar Gonz\'alez-Cervantes: https://orcid.org/0000-0003-4835-5436
\\
Juan Eduardo Napoles-Valdes: https://orcid.org/0000-0003-2470-1090

\end{document}